\documentclass[11pt,A4paper]{article}

\usepackage[left=32mm,right=32mm,top=30mm,bottom=32mm]{geometry}
\usepackage{labelfig}
\usepackage{epsfig}
\usepackage{epstopdf}

\usepackage{xfrac}

\usepackage{color}
\usepackage{amsthm,amsmath,amssymb}
\usepackage{booktabs}
\usepackage{mathpazo}
\usepackage{microtype}
\usepackage{overpic}
\usepackage{bm}
\usepackage{sectsty}
\usepackage[
	pdftitle={PDFTitle},
	pdfauthor={Hugo Parlier},
	ocgcolorlinks,
	linkcolor=linkred,
	citecolor=linkred,
	urlcolor=linkblue]
{hyperref}

\definecolor{linkred}{rgb}{0.48,0.1,0.05}
\definecolor{linkblue}{RGB}{16, 78, 139}

\usepackage[hang,flushmargin]{footmisc}
\usepackage{enumitem}
\usepackage{titlesec}
	\titlespacing{\section}{0pt}{12pt}{0pt}
	\titlespacing{\subsection}{0pt}{6pt}{0pt}
	
\titlelabel{\thetitle.\quad}

\makeatletter 

\long\def\@footnotetext#1{%
\H@@footnotetext{%
\ifHy@nesting 
\hyper@@anchor{\@currentHref}{#1}%
\else 
\Hy@raisedlink{\hyper@@anchor{\@currentHref}{\relax}}#1%
\fi 
}}

\def\@footnotemark{%
\leavevmode 
\ifhmode\edef\@x@sf{\the\spacefactor}\nobreak\fi 
\H@refstepcounter{Hfootnote}%
\hyper@makecurrent{Hfootnote}%
\hyper@linkstart{link}{\@currentHref}%
\@makefnmark 
\hyper@linkend 
\ifhmode\spacefactor\@x@sf\fi 
\relax 
}%

\ifFN@multiplefootnote%
\renewcommand*\@footnotemark{%
\leavevmode 
\ifhmode 
\edef\@x@sf{\the\spacefactor}%
\FN@mf@check 
\nobreak 
\fi 
\H@refstepcounter{Hfootnote}%
\hyper@makecurrent{Hfootnote}%
\hyper@linkstart{link}{\@currentHref}%
\@makefnmark 
\hyper@linkend 
\ifFN@pp@towrite 
\FN@pp@writetemp 
\FN@pp@towritefalse 
\fi 
\FN@mf@prepare 
\ifhmode\spacefactor\@x@sf\fi 
\relax%
}%
\fi 

\makeatother 

\theoremstyle{plain}
\newtheorem{theorem}{Theorem}[section]

\newtheorem{lemma}[theorem]{Lemma}
\newtheorem{corollary}[theorem]{Corollary}
\newtheorem{conjecture}[theorem]{Conjecture}

\theoremstyle{definition}

\newtheorem{notation}[theorem]{Notation}

\newtheorem{remark}[theorem]{Remark}

\newcommand{\sys}{{\rm sys}}
\newcommand{\ii}{{\mathit i}}

\newcommand{\Z}{{\mathbb Z}}

\newcommand{\M}{{\mathcal M}}
\newcommand{\PSL}{{\rm PSL}}

\newcommand{\arcsinh}{{\,\rm arcsinh}}

\newcommand{\numS}{{\rm Kiss}(S)}

\newcommand{\curves}{\mathcal H}

\sectionfont{\large \sc}
\subsectionfont{\normalsize}

\long\def\symbolfootnote[#1]#2{\begingroup%
\def\thefootnote{\fnsymbol{footnote}}\footnote[#1]{#2}\endgroup}

\def\blfootnote{\xdef\@thefnmark{}\@footnotetext}

\begin{document}

{\Large \bfseries \sc Kissing numbers for surfaces}

{\bfseries Hugo Parlier\symbolfootnote[2]{\normalsize Research supported by Swiss National Science Foundation grant number PP00P2\textunderscore 128557\\ 
{\em Address:} Department of Mathematics, University of Fribourg, Switzerland \\
{\em Email:} \href{mailto:hugo.parlier@unifr.ch}{hugo.parlier@unifr.ch}
}}

{\em Abstract.} The so-called {\it kissing number} for hyperbolic surfaces is the maximum number of homotopically distinct systoles a surface of given genus $g$ can have. These numbers, first studied (and named) by Schmutz Schaller by analogy with lattice sphere packings, are known to grow, as a function of genus, at least like $g^{\sfrac{4}{3}-\varepsilon}$ for any $\varepsilon >0$. The first goal of this article is to give upper bounds on these numbers; in particular the growth is shown to be sub-quadratic. In the second part, a construction of (non hyperbolic) surfaces with roughly $g^{\sfrac{3}{2}}$ systoles is given.

\vspace{1cm}

\section{Introduction}
The classical kissing number problem for sphere packings is the search for an optimal upper bound on the number of $n$-dimensional (euclidean) unit spheres, pairwise disjoint in their interior, that can be tangent to a fixed unit sphere. Exact values for these numbers, commonly called kissing numbers for simplicity, are only known in a finite number of cases ($n=1-4, 8, 24$, see \cite{pfzi04} and references therein). A seemingly easier problem is to ask the same question as above, but with the restriction that the centers of the spheres lie on some euclidean lattice. This gives rise to the so-called kissing number problem for lattice sphere packings. As above, values are only known in a finite number of cases ($n=1-9, 24$, see \cite{coslbook} and references therein). In some cases the solutions to the two problems coincide (but not in general, for instance they are known to be different for $n=9$). As an example, when $n=2$, the answer is $6$ for both problems and the optimal solution is given by points lying on the hexagonal lattice. In this case, the question is equivalent to asking how many shortest (nontrivial) vectors a lattice can contain. In terms of {\it systoles} (meaning shortest non-trival curves of a manifold), how many distinct isotopy classes of systoles can a flat torus have? The kissing number problem for lattices is the natural generalization: how many isotopy classes of systoles can the underlying $n$-torus contain?

Another possible generalization - proposed by Schmutz Schaller - is to hyperbolic closed surfaces. We ask the same question where genus plays the part of dimension. For hyperbolic surfaces there is a unique geodesic in a prescribed nontrivial isotopy class, so the question is as follows.

{\it How many systoles can a hyperbolic surface of genus $g$ have?}

By analogy with the lattice case, Schmutz Schaller defined the maximum number for each genus to be the kissing number for hyperbolic surfaces. Usually, we consider simple closed geodesics, and systoles in particular, to be unoriented objects, so morally we should be dividing lattice sphere packing kissing numbers by two to make the two problems truly analogous. As in the case of lattices, very little is known about exact values: the only known case is for genus $2$ where the answer is $12$. 

A related problem for systoles of hyperbolic surfaces is to the search for surfaces with systole of maximal length, again initiated by Schmutz Schaller \cite{sc931}. As before, there is a possible analogy with lattices: this is the hyperbolic surface version of the search for the optimal Hermite constants. (The Hermite constant can be thought of as the square of the maximal systole length of unit volume flat torus.) For each genus, via a compactness argument, there exists a hyperbolic surface with longest possible systole. Again, the only known exact value is for genus $2$, a result of Jenni \cite{je84}. It is realized by the same surface as for the kissing number problem - the so-called Bolza surface. The Bolza surface is also maximal among genus $2$ hyperbolic surfaces for the number of self-isometries. In the non-compact cases, arithmetic surfaces coming from principal congruence subgroups of $\PSL_2(\Z)$ are known to be maximal among finite area hyperbolic surfaces in their respective moduli spaces, a result of Schmutz Schaller \cite{sc941}, see also \cite{ad98} for another proof. For both problems, although finding exact solutions seems difficult, one can ask for upper and lower bounds, and in particular it is natural to study the asymptotic growth of how these constants vary in function of genus. Via a simple area argument, it is easy to see that the systole length of a hyperbolic surface cannot exceed $2 \log g$ ($+C$ for some constant $C>0$). Buser and Sarnak \cite{busa94} were the first to construct families of surfaces with $\sim \log g$ systole growth. More precisely, they showed
$$
\limsup_{g\to \infty}\frac{ \max_{S\in \M_g} \sys(S)}{ \log g}\geq \frac{4}{3}
$$
where $\M_g$ is the moduli space of all hyperbolic surfaces up to isometry and $\sys(S)$ is the length of the systole of $S$. Since then there have been other constructions, see for example \cite{kascvi07}, but all are based on arithmetic methods of some sort. The existence of such surfaces is somewhat surprising: the radius of a maximally embedded disk in a hyperbolic surface is $\sim \log g$, so a surface with $\sim \log g$ systole looks essentially everywhere like a fat disk that is pasted together in some clever way in order to avoid creating a short loop somewhere.

Schmutz Schaller wrote a series of papers for the hyperbolic surface kissing number problem, in both the compact and non-compact cases, where he proved some interesting lower bounds. In first instance, one might think that it might not be possible to have a surface where the number of systoles is considerably bigger than the size of a maximal isometry group. Via Hurwitz's bound, this would imply an upper bound to the number of systoles which grows linearly in genus. Schmutz Schaller's first result on this was the existence of families of surfaces with a number of systoles that grew more than linearly in genus \cite{sc962,sc964}. The best results for closed surfaces appeared later \cite{sc97} where he showed that
$$
\limsup_{g\to \infty} \frac{ \log \left(\max_{S\in \M_g} \numS\right)}{\log g} \geq \frac{4}{3}-\varepsilon
$$
for any $\varepsilon>0$ where $\numS$ denotes the number of systoles of $S$. Stated otherwise, for any $\varepsilon>0$, there exists a family of surfaces, one in each genus, with more than $g^{\sfrac{4}{3}-\varepsilon}$ systoles (for large enough $g$). Again, the construction is based on arithmetic methods. Based on the intuition that one cannot do better for arithmetic surfaces  and that arithmetic surfaces should be optimal for these problems, Schmutz Schaller \cite{sc98} made two conjectures, namely that the inequalities above are in fact equalities (with the right hand side equal to $\sfrac{4}{3}$). For the sake of avoiding repititious repetition, we rephrase them as follows.
\begin{conjecture}\label{conj:size}
There exists a constant $A>0$ such that
$$
\max_{S\in \M_g} \sys(S) \leq \frac{4}{3} \log g +A.
$$
\end{conjecture}
\begin{conjecture}\label{conj:number}
There exists a constant $B>0$ such that
$$
\max_{S\in \M_g} \numS \leq B g^{\sfrac{4}{3}}.
$$
\end{conjecture}
The two problems seem to run in parallel because of the arithmetic nature of the interesting known examples, but Schmutz Schaller did not establish a direct link between the two. In contrast to Conjecture \ref{conj:size}, there was no ``easy" known upper bound for Conjecture \ref{conj:number} that behaves roughly as conjectured. One might expect an easy quadratic upper bound for the number of systoles which would seem to be the counterpart to the straightforward $2 \log g$ bound.  (In fact Schmutz Schaller claims such a quadratic bound in \cite{sc942} but the argument, supposedly only based on the topological condition that two systoles pairwise intersect at most once, is faulty and does not seem to be easily repairable.)

The first result of this paper is to find an upper bound for $\numS$, which depends on the length of the systole of $S$.
\begin{theorem}\label{thm:mainhyperbolicintro}
There exists a constant $U>0$ such that for any hyperbolic surface $S$ of genus $g$ with systole $\ell$ the following holds:
$$
\numS \leq  U\, \frac{e^{\,\ell/2}}{\ell} \,g.
$$
\end{theorem}
Using the $2\,\log g$ bound on $\ell$ mentioned above, this has the following consequence.
\begin{corollary}\label{cor:subquad}
There exists a constant $U>0$ such that any hyperbolic surface of genus $g$ has at most $U \dfrac{g^2}{\log g}$ systoles.
\end{corollary}
More generally, if there exists an upper bound on length of systoles of type $C \log g$ for some $C>0$, then Theorem \ref{thm:mainhyperbolicintro} implies a bound of order $\sim g^{\sfrac{C}{2} + 1}$. In particular a positive answer to Conjecture \ref{conj:size} would imply that there are at most $\sim g^{\sfrac{5}{3}}$ systoles on a genus $g$ surface. Another consequence is that Conjecture \ref{conj:number} holds for all surfaces with systole bounded above by $\sfrac{2}{3} \log g$.

The bound also shows that if a family of surfaces has sub-logarithmic systole growth, then the number of systoles is ``almost" at most linear. More precisely:
\begin{corollary}\label{cor:linear}
Let $f(g)$ be any positive function with $\lim_{g\to \infty} \sfrac{f(g)}{\log g} = 0$. Then
$$
\max \{\numS \,|\, S\in \M_g{\text{ with }}\sys(S)\leq f(g)\} \leq g^{1+\varepsilon}.
$$
for any $\varepsilon>0$ and large enough $g$.
\end{corollary}
The theorem implies more accurate than the above corollary for ``intermediate" growth, but the formulation above is given for clarity. In particular, this means that any family of surfaces with ``many" systoles (by which we mean at least $g^{1+a}$ for some $a>0$), then the family has $\sim \log g$ systole growth as well. Of course it is not a priori easier to construct surfaces with many systoles, but if one does, then the large size systoles come for free.

In the more general context of Riemannian metrics on surfaces, similar questions can be asked. One no longer has uniqueness of a geodesic in an isotopy class, so the appropriate question on the number of systoles is an upper bound on the number of distinct isotopy classes of simple closed curves that can simultaneously be realized as systoles for some metric. 

{\it How many systoles up to isotopy can a closed Riemannian surface of genus $g$ have?}

We'll relate to optimal upper bounds to this problem as the kissing number problem for general surfaces.

It is not known what sets of topological curves on either a hyperbolic or Riemannian surface can be realized by systoles. By a cutting and pasting argument, it is not too difficult to see that on any closed surface (not necessarily hyperbolic) a systole is necessarily a simple closed curve. Likewise, any two systoles can pairwise intersect at most once. This leads us to a related purely topological problem due to B. Farb and C. Leininger (see \cite{marith10}).

{\it Up to isotopy, how many distinct curves can be realized on a surface of genus $g$ such that they pairwise intersect at most once?}

It seems to be a surprisingly hard question to answer. The best known upper bound is in fact exponential, and the best lower bound quadratic \cite{marith10}. These numbers provide upper bounds for the kissing number problems, and to the best of the author's knowledge, these are the best known upper bounds for kissing numbers of general surfaces and were the best bounds even in the case of hyperbolic surfaces prior to Corollary \ref{cor:subquad}. Observe that Corollary \ref{cor:subquad} shows that optimal kissing numbers for hyperbolic surfaces cannot be the same as the numbers coming from the purely topological problem. It was already known that the topological condition was quite different from the systolic condition: in \cite{paanpesys}, there are constructions of configurations of isotopy classes curves that fail to be systoles for {\it any} Riemannian metric on the surface.

For completeness, we mention that the problem of finding an upper bound on systole length in the case of variable curvature can also be made to make sense. For a genus $g$ surface with area normalized to the area of its hyperbolic counterparts, Gromov \cite{gr83,gr96} gave a $\sim \log g$ upper bound, which in light of the hyperbolic examples is roughly optimal, but the precise asymptotic growth remains completely open. Whether the asymptotic growth for both problems should be different for both the upper bound on length or the kissing numbers is also unknown.

The second main result of this paper is about lower bounds for kissing number for Riemannian surfaces.
\begin{theorem}\label{thm:generalkiss}
There exist surfaces of genus $g>0$ with a number of systoles of order of growth at least $g^{\sfrac{3}{2}}$.
\end{theorem}
This shows that if Conjecture \ref{conj:number} is correct, the asymptotic growth for kissing numbers is considerably different in the case of variable curvature surfaces. Recall the best known - and conjectured optimal - lower bounds for hyperbolic surfaces are roughly $g^{\sfrac{4}{3}}$. 

The proof is by construction, and the geometry of the surfaces come from embeddings of complete graphs. One might wonder if one can use these techniques to find hyperbolic surfaces with the same behavior, but the surfaces are far from being hyperbolic. In fact, in striking contrast to what is possible in the hyperbolic case (Corollary \ref{cor:linear}), the systole length is quite small proportionally to area (constant systole length for area $\sim g$).

\section{Bounds on numbers of systoles of hyperbolic surfaces}

We denote $\numS$ the number of systoles of a surface $S$. Before proceeding to proofs of upper bounds on $\numS$ for hyperbolic surfaces, we begin with some observations on the geometry of systoles.

\subsection{Geometric properties of systoles}

The estimates needed rely essentially on trigonometric arguments in the hyperbolic plane. Recall the classical collar lemma.

\begin{lemma}[Collar lemma]
Let $\gamma$ be a simple closed geodesic of length $\ell$. Then there is an embedded collar of width $w(\ell)$ around $\gamma$ where
$$
w(\ell)= \arcsinh\left(\frac{1}{\sinh(\ell/2)}\right).
$$
Furthermore, any simple closed geodesic $\delta$ that enters this collar essentially intersects $\gamma$.
\end{lemma}
For systoles, one can do even better.

\begin{lemma}[Systolic collar lemma]\label{lem:scl}
Let $\alpha$ and $\beta$ be systoles of length $\ell$ that don't intersect. Then there are at a distance at least $2 r(\ell)$ where
$$
r(\ell) =  \arcsinh\left(\frac{1}{2\,\sinh(\ell/4)}\right).
$$
\end{lemma}
\begin{proof}
Take a shortest path $c$ between $\alpha$ and $\beta$ and consider the geodesic $\gamma$ in the free homotopy class of $\alpha * c * \beta * c^{-1}$ where $\alpha,\beta$ and $c$ are oriented so that $\gamma$ is simple (see figure \ref{fig:orientedpants}). Note that $\alpha,\beta,\gamma$ form a pair of pants.
\vspace{-6pt}
\begin{figure}[h]
\leavevmode \SetLabels
\L(.37*.92) $\alpha$\\
\L(.62*.92) $\beta$\\
\L(.49*.11) $\gamma$\\
\L(.49*.6) $c$\\
\endSetLabels
\begin{center}
\AffixLabels{\centerline{\epsfig{file =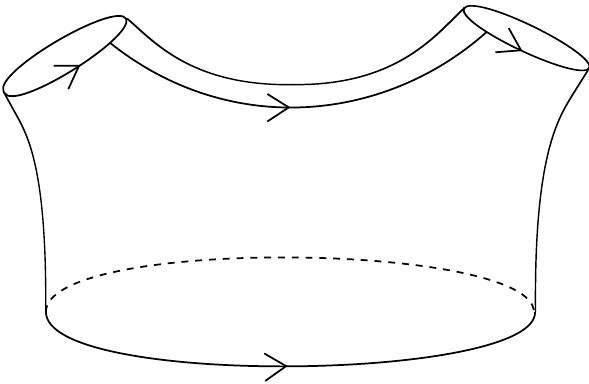,width=4.5cm,angle=0}}}
\vspace{-24pt}
\end{center}
\caption{Orientations of $\alpha$, $\beta$ and $c$} \label{fig:orientedpants}
\end{figure}
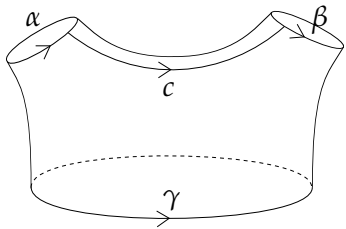
Now because $\ell(\gamma)\geq \ell$, we can compute the minimal distance between $\alpha$ and $\beta$. The result follows from standard trigonometry in a pair of pants and the double $\sinh$ formula.
\end{proof}

From this the following corollary is  immediate.

\begin{corollary}\label{cor:close} If $\gamma$ and $\delta$ are two systoles of length at most $\ell$ which pass through a same disk of radius $r(\ell)$, then they essentially intersect.
\end{corollary}

Using only this observation, it is possible to obtain a universal polynomial bound on the number of systoles of hyperbolic surfaces, but we need to work harder to obtain our estimates.

We begin by noticing that if two systoles intersect, then their angle of intersection can be bounded below by their length.

\begin{lemma}\label{lem:angle} Let $\gamma$ and $\delta$ be systoles of length $\ell$ that intersect. Then $\sin \angle(\gamma,\delta) > \frac{1}{2} \left(\cosh(\ell/4)\right)^{-1}. $
\end{lemma}

\begin{proof}
Let $p$ be the intersection point of the two curves and consider the point $q$ of $\delta$ that is at distance exactly $\ell/2$ away from $p$. Fix one of the two paths of $\delta$, say $\delta_1$, between $q$ and $\gamma$. Among all paths freely homotopic to $\gamma_1$ with one fixed endpoint at $q$ and the other on any point of $\gamma$, there is one of minimal length which we shall denote $h$. Denote $q'$ the point of intersection of $h$ and $\gamma$. We now have a right triangle $p,q,q'$ with hypothenuse $H$ is of length $\sfrac{\ell}{2}$, the side $h$, and the basis $b$. Observe that using the other arc of $\delta$, say $\delta_2$, one obtains a symmetric situation with an isometric right angle triangle where the two triangles are linked via a rotation of angle $\pi$ around $q$ (see figure \ref{fig:anglepants}). 

\begin{figure}[h]
\leavevmode \SetLabels
\L(.13*.466) $\gamma$\\
\L(.19*.33) $p$\\
\L(.325*.49) $q$\\
\L(.26*.73) $\delta$\\
\L(.52*.69) $p$\\
\L(.67*.58) $q$\\
\L(.534*.865) $q'$\\
\L(.805*1.04) $p$\\
\L(.49*.9) $\gamma$\\
\L(.88*.937) $\gamma$\\
\endSetLabels
\begin{center}
\AffixLabels{\centerline{\epsfig{file =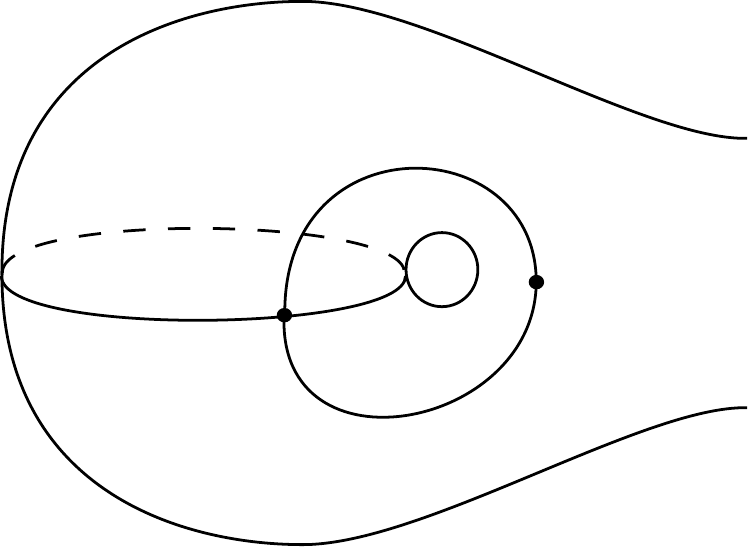,height=3.5cm,angle=0}\hspace{1.0cm} \epsfig{file =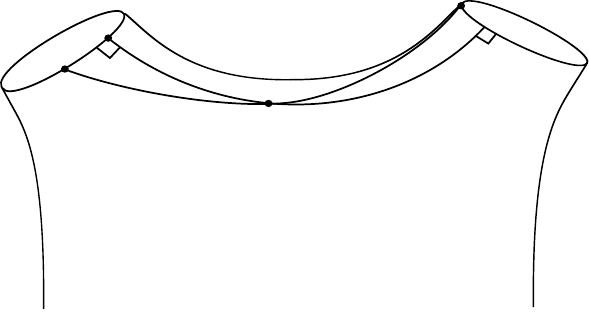,height=3.5cm,angle=0}}}
\vspace{-18pt}
\end{center}
\caption{Two intersecting systoles} \label{fig:anglepants}
\end{figure}
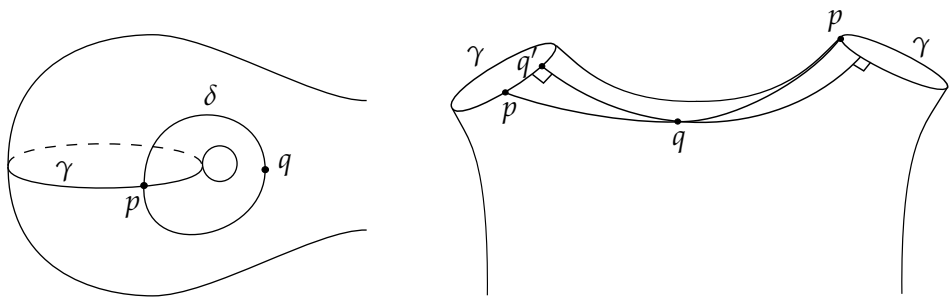
As the two triangles are isometric, we can concentrate on the first one. It is not difficult to see that $b$ is of length at most $\ell/4$, otherwise one can construct a shorter nontrivial path than $\delta$. 

Consider the angle $\theta=\angle(\gamma,\delta)$. By the sine formula for hyperbolic triangles, it satisfies
$$
\frac{\sin(\theta)}{\sinh(h)} = \frac{\sin(\pi/2)}{\sinh(\ell/2)}.
$$
Now because the basis is at most $\ell/4$, $h$ is strictly greater than $\ell/4$ (otherwise the same occurs in the other triangle, and we obtain $\ell(\delta)<\ell$). The inequality 
$$
\sin \angle(\gamma,\delta) > \frac{\sinh(\ell/4)}{\sinh(\ell/2)} 
$$
and one concludes by using the $\sinh$ formula for a double angle.
\end{proof}

\begin{notation} We denote $\theta_\ell := \arcsin \frac{1}{2 \cosh(\ell/4)}$.
\end{notation}

For large $g$, the behavior is roughly $e^{\ell/4}$. In particular, using the $2 \log g$ upper bound on length, this implies a $\frac{1}{\sqrt{g}}$ lower bound on the angle between systoles. Thus a collection of systoles that intersect in a single point cannot have cardinality greater than roughly $\sqrt{g}$. This is very different from the case of variable curvature where one can construct surfaces with $\sim g$ systoles that intersect in a single point, see Remark \ref{rem:gint}.

Corollary \ref{cor:close} ensures that two systoles $\gamma,\delta$ with points $p_\gamma \in \gamma$, $p_\delta \in \delta$ that satisfy $d(p_\gamma,p_\delta)< d(\ell)$ must intersect somewhere on the surface. This next lemma gives a bound on how far the intersection point can be from $p_\gamma$ and $p_\delta$. 

\begin{lemma}\label{lem:dist}
Let $\gamma,\delta$ be systoles of length $\ell$ which cross a disk of radius $r(\ell)$ with center $p$. Then the intersection point $q$ between $\gamma$ and $\delta$ satisfies
$$
d(p,q) < \arcsinh\left(2 \coth\frac{\ell}{4}\right).
$$
\end{lemma}

\begin{proof}
Consider points $p_\gamma \in \gamma$, $p_\delta \in \delta$ which lie in the disk centered in $p$. Consider the two angles $\angle(q,p,p_\gamma)$ and $\angle(q,p,p_\delta)$. By Lemma \ref{lem:angle}, one of the two angles must be greater than $\theta_\ell\over 2$. Without loss of generality, let us suppose that this angle is $\theta=\angle(q,p,p_\gamma)$. We note that via Lemma \ref{lem:angle} again we obtain

$$
\sin(\theta) \geq \sin({\theta_\ell / 2}) = \frac{1}{2 \cos (\theta_\ell / 2)} \sin(\theta_\ell) >  \frac{1}{4 \cosh(\ell/4)}.
$$

We now concentrate our attention to the triangle $p,q,p_\gamma$. Denote by $\theta'$ the angle opposite the side $\overline{pq}$ (see figure \ref{fig:disktriangles}). We have the following identity:
$$
\frac{\sin(\theta')}{\sinh d(p,q)}= \frac{\sin(\theta)}{\sinh d(p,p_\gamma)}.
$$

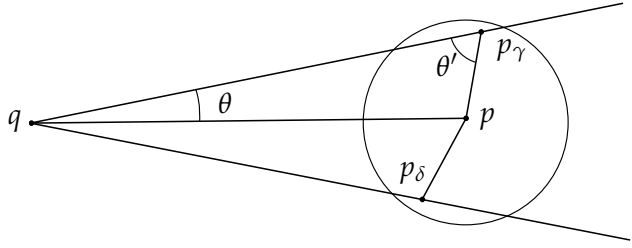
\begin{figure}[h]
\leavevmode \SetLabels
\L(.218*.49) $q$\\
\L(.593*.693) $\theta'$\\
\L(.402*.537) $\theta$\\
\L(.63*.5) $p$\\
\L(.645*.79) $p_\gamma$\\
\L(.56*.27) $p_\delta$\\
\endSetLabels
\begin{center}
\AffixLabels{\centerline{\epsfig{file =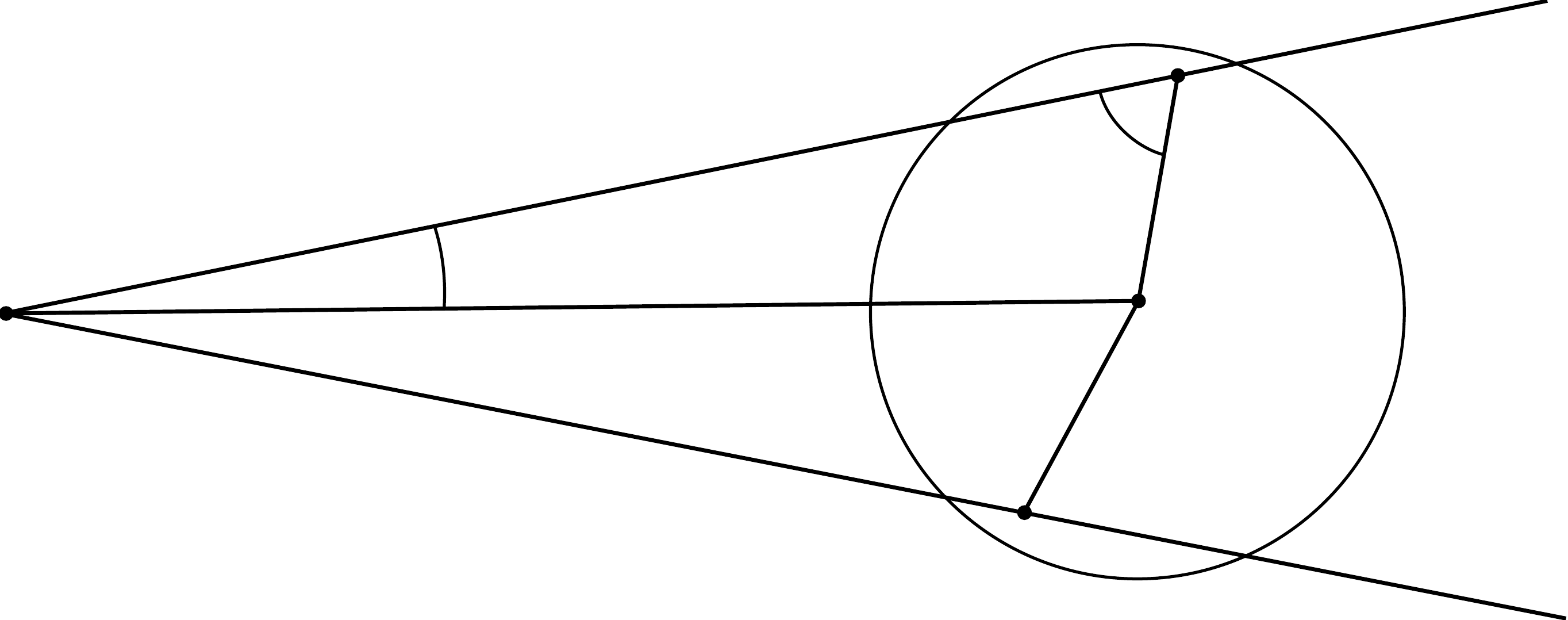,width=8.0cm,angle=0}}}
\vspace{-18pt}
\end{center}
\caption{Computing the distance to an intersection point} \label{fig:disktriangles}
\end{figure}

Via Lemma \ref{lem:scl} and what precedes we have 
$$
\sinh d(p,q) < \frac{\sin(\theta')}{\sin(\theta)} \frac{1}{2\,\sinh(\ell/4)} < 2 \coth \frac{\ell}{4}$$
which concludes the proof.
\end{proof}

\begin{notation} For future reference, we fix the following notation:
$$
R(\ell):=\arcsinh\left(2 \coth\frac{\ell}{4}\right).
$$
\end{notation}

\begin{remark}\label{rem:radius}
Observe that via the same argument, for two systoles that intersect in a point $q$ and both simultaneously pass through a disk of radius $\rho$ and center $p$, the following holds:
$$
\sinh d(p,q) \leq \frac{\sinh \rho}{\sin(\theta_\ell/2)} <  \frac{\sinh \rho}{4 \cosh(\ell/4)}.
$$
\end{remark}

\subsection{Proof of the upper bound} We can now proceed to the proof of our main upper bound. We begin by giving a more precise version of Theorem \ref{thm:mainhyperbolicintro}.

\begin{theorem}\label{thm:mainhyperbolic}
Let $S$ be a hyperbolic surface of genus $g$ with systole bounded by $\ell$. Then
$$
\numS \leq C_\ell \,(g-1)
$$
where $C_\ell$ is a constant that depends on $\ell$ which can be taken to be
$$
C_\ell = 100\,  \frac{e^{\,\sfrac{\ell}{2}}}{\ell}.
$$
\end{theorem}

\begin{remark}
The constant $100$ in front can easily be improved, but it is the order of growth we are interested in.
\end{remark}

\begin{proof}
If $L\leq 2 \arcsinh 1$ then, by the collar lemma, systoles are disjoint, and thus, as there are at most $3g-3$ disjoint simple closed geodesics on a genus $g$ surface, we obtain
$$
\numS \leq 3g-3.
$$
We can now concentrate on the case when $\ell \geq 2 \arcsinh 1$.

The basic strategy will be the following. We begin by covering the surface $S$ by balls of radius $r(\ell)$ (where $r(\ell)$ is given by Lemma \ref{lem:scl}). The first step will be to estimate $F(S)$, an upper bound on the minimum number of these balls required to cover $S$. We'll then find an upper bound $G(S)$ on the number of systoles that can cross such a ball. Finally, if we denote $H(S)$ the minimum number of covering balls that a systole of $S$ must cross, we have 
$$
\numS \leq \frac{F(S)\, G(S)}{H(S)}.
$$
Let us now concentrate on finding bounds for these quantities in function of $\ell$ and $g$.\\

\noindent{\underline{The number of balls required to cover $S$}}

As usual in these type of estimates, we use 

\begin{center}
$
F(S)=\{\text{Number of balls of radius $r(\ell)$ needed to cover $S$}\}$

$\leq$

$ \{\text{Max number of balls of radius $\sfrac{r(\ell)}{2}$ that embed and are pairwise disjoint}\}.$
\end{center}
Now as
$$
\text{Area}(D_{\sfrac{r(\ell)}{2}}) = 2\pi(\cosh(\sfrac{r(\ell)}{2})-1) = 2\pi\left( \cosh\left(\frac{\arcsinh\left(\frac{1}{2\sinh(\ell/4)}\right)}{2}\right) -1 \right)
$$
we deduce a bound for $F(S)$ that depends only on $g$ (coming from the area of $S$) and $\ell$:
$$
F(S) \leq \frac{\text{Area}(S)}{\text{Area}(D_{\sfrac{r(\ell)}{2}})}.
$$
To get an idea of the order of growth of this bound, observe that
$$
\text{Area}(D_{\sfrac{r(\ell)}{2}}) > \frac{\pi}{4} e^{-\sfrac{\ell}{2}}
$$ 
and it follows that
$$
F(S) <16 (g-1) \, e^{\sfrac{\ell}{2}}.
$$
\vspace{0.3cm}

\noindent{\underline{The number of systoles intersecting each ball}}

We now proceed to find an upper bound on the number of systoles that can intersect a ball of radius $r(\ell)$. Consider a disk $D_0$ on $S$ of radius $r(\ell)$. Any two given systoles that cross $D_0$ must intersect, and via Lemma \ref{lem:dist} we know that their intersection point $q$ lies within $R(\ell)$ of the center $p$ of $D_0$. We shall now reason in the universal cover, considering $D_0$ and the disk $D_1$ of center $p$ and radius $R(\ell)$. The geometric problem we are interested in is as follows.

{\it How many (hyperbolic) lines, any two of which pairwise intersect in an angle of at least $\theta_\ell$, can intersect $D_0$?}

For this consider the disk $D_2$, also of center $p$, but of radius $R(\ell)+R'$ for a given $R'>0$. Consider geodesics $\gamma_1$ and $\gamma_2$ which pass through $D_0$ and intersect in $q \in D_1$. From $q$, we consider one of the four angular sectors of angle $\theta>0$ and the two rays of $\gamma_1$ and $\gamma_2$ which bound the sector. We set $q_1,q_2$ to be the intersection points of the rays of $\gamma_1$ and $\gamma_2$ with the boundary of $D_2$ as in figure \ref{fig:3disks}. Note that $d(q,q_1)$ and $d(q,q_2) > R'$ and that $\theta> \theta_0$.

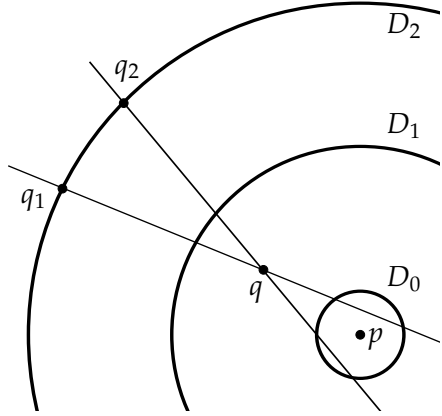
\begin{figure}[h]
\leavevmode \SetLabels
\L(.633*.166) $p$\\
\L(.53*.28) $q$\\
\L(.65*.30) $D_0$\\
\L(.65*.67) $D_1$\\
\L(.65*.91) $D_2$\\
\L(.33*.5) $q_1$\\
\L(.412*.813) $q_2$\\
\endSetLabels
\begin{center}
\AffixLabels{\centerline{\epsfig{file =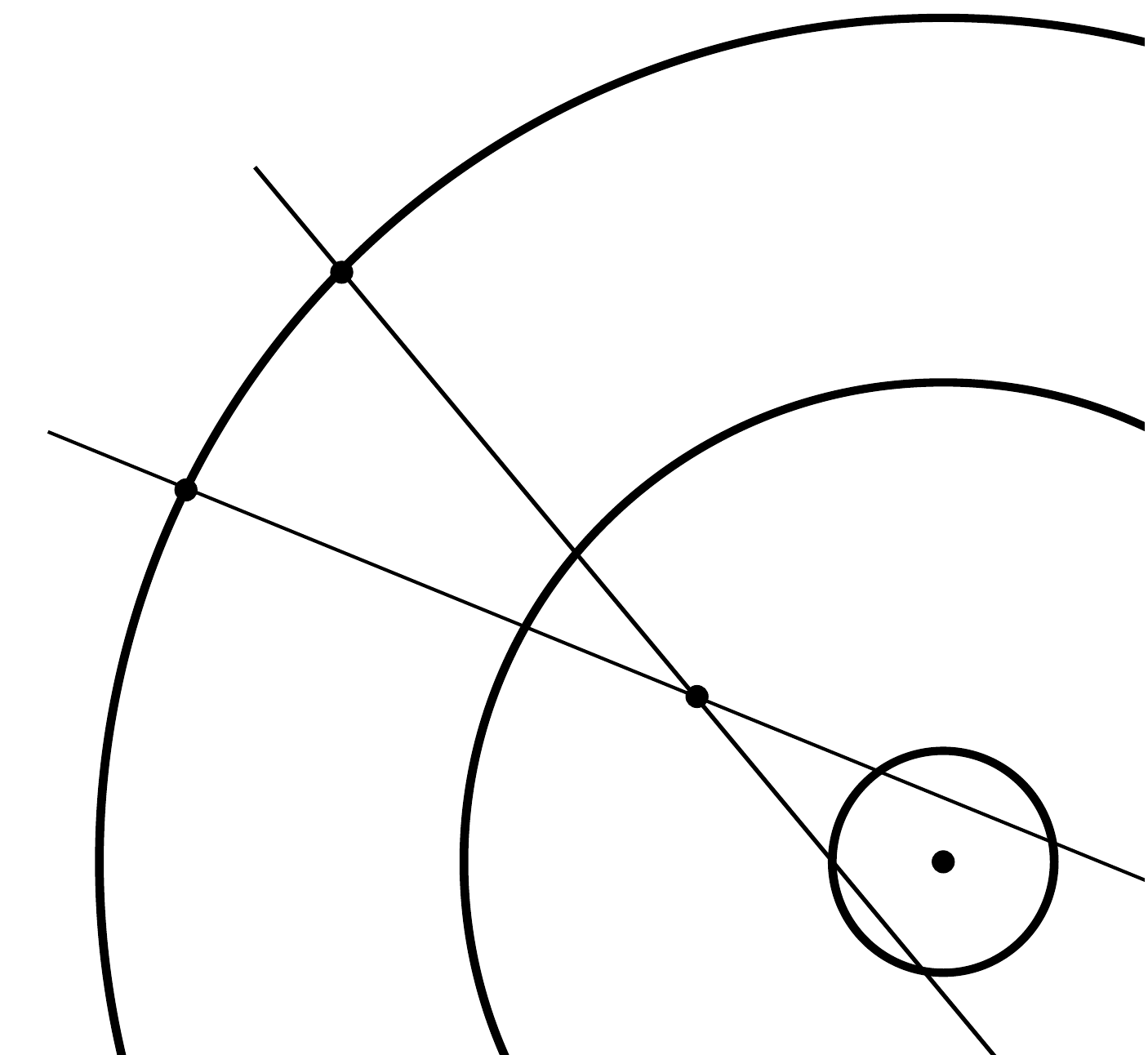,width=6.0cm,angle=0}}}
\vspace{-18pt}
\end{center}
\caption{The disks $D_0, D_1$ and $D_2$} \label{fig:3disks}
\end{figure}

Considering the triangle $q,q_1,q_2$, observe the following: the distance $d(q_1,q_2)$ is strictly greater than the distance one would compute if 
$$d(q,q_1)=d(q,q_2) = R' {\text{ and }} \theta=\theta_0.$$
Via hyperbolic trigonometry in the $(q,q_1,q_2)$ triangle:
$$
d(q_1,q_2) > 2 \arcsinh\left( \sinh(R') \sin(\theta_\ell)\right).
$$
We now consider a maximal set of lines $\gamma_1,\hdots,\gamma_k$ that cross $D$. The lines divide $\partial D_2$ into $2k$ arcs. As such, the cardinality of the set of lines is bounded by the length of $\partial D_2$ divided by twice the minimal distance between the intersection points of the lines with $D_2$. By the above estimates, we have
$$
G(S) < \frac{ \pi \sinh(R(\ell)+R')}{2 \arcsinh(\sinh(R') \sinh(\theta_\ell))}.
$$
Setting $R' := \arcsinh(1)$ this becomes
$$
G(S) < \frac{\pi}{2} \frac{\sinh(R(\ell) +\arcsinh(1))}{\arcsin \frac{1}{2 \cosh(\ell/4)}}.
$$
Note that $R(\ell)$ remains bounded as $\ell$ increases so we get an upper bound on $G(S)$ which has order of growth $\sim e^{\sfrac{\ell}{4}}$.

\vspace{0.5cm}

\noindent{\underline{The number of balls each systole crosses}}

\noindent The last estimate we need is a lower bound on the number of balls $H(S)$ in our covering of $S$ that each systole necessarily crosses. The estimate we'll use is straightforward: to cover a geodesic segment of length $\ell$ with balls of radius $r(\ell)$ one requires at least $2\, \sfrac{\ell}{r(\ell)}$ balls. Thus
$$
H(S)\geq \frac{2\, \ell}{r(\ell)} =  \frac{2\, \ell}{\arcsinh\left(\frac{1}{2\sinh(\ell/4)}\right)}.
$$
Here the order of growth is roughly $\ell \,e^{\,\sfrac{\ell}{4}}$.

We can now conclude by using our estimates for $F(S),G(S)$ and $H(S)$. By our above estimates, the order of growth of the bound is 
$$\sim \text{Area}(S) \frac{e^{\,\sfrac{\ell}{2}}}{\ell}.
$$
By a few elementary considerations, for $\ell \geq 2 \arcsinh 1$, we have the following effective inequality
$$
\numS < 100 \,(g-1) \frac{e^{\,\sfrac{\ell}{2}}}{\ell}
$$
which concludes the proof.
\end{proof}

The proofs of Corollaries \ref{cor:subquad} and \ref{cor:linear} in the introduction follow from elementary estimates.

\section{Non-hyperbolic surfaces with many systoles}
In this section, we construct non-hyperbolic surfaces with many systoles. Recall the statement from the introduction.
\begin{theorem}\label{thm:completegraph}
There exist surfaces of genus $g>0$ with a number of systoles of order of growth at least $g^{\sfrac{3}{2}}$.
\end{theorem}
\begin{remark}
The statement is stated in terms of order of growth for simplicity. The result we shall in fact prove is stronger, namely that there exist surfaces with at least
$$
g\sqrt{48 g - 47} +15 g +\frac{1}{3} \sqrt{48 g - 47} - \frac{41}{3} > 6 g^{\sfrac{3}{2}}
$$
number of systoles, but it is the order of growth we are really interested in.
\end{remark}

\begin{proof}
The proof is by construction. The general idea of the proof is as follows: we'll begin with a complete graph $\Gamma_n$ (with $n>3$ for this construction to work) and imitate its geometry on a surface into which it is embedded. To give a rough idea of where the numbers are coming from, recall that one can embed a complete $n$-graph into a surface of genus $g\sim n^2$. On such a graph, the any nontrivial loop has length at least $3$, so the short loops are those of length exactly $3$ and there are $\left(
\begin{array}{c}
n  \\
3  
\end{array}
\right)
 \sim n^3 \sim g^{\sfrac{3}{2}}$ of those. The idea is then to construct an (almost) isometric embedding so that ``many" of the nontrivial short cycles of the graph remain nontrivial on the surface, and so that there aren't any others of smaller length. Then the surface will require some tweaking so that the lengths of the short nontrivial curves are all exactly the same.

Our first observation is topological: consider a minimal genus surface $S_{g_n}$ of genus $g_n$ into which $\Gamma_n$ is embedded. By Ringel's and Youngs' Theorem \cite{riyo68} we have
$$
g_n < \frac{(n-3)(n-4)}{12} +1.
$$
This embedding will serve as our blueprint for the construction of the surface. In fact, we will begin building the geometry of our goal surface in a neighborhood of the embedded graph before describing the metric on the full surface.

Geometrically we will be thinking of our graph as a metric graph with all edges of length $1$, and obtained by pasting $n$-pods together (each $n$-pod consists of a vertex with $n$ half-edges of length $\sfrac{1}{2}$). We begin by embedding each $n$-pod in the euclidean plane as follows. The vertex is on the origin and the half-edges are euclidean segments of length $\sfrac{1}{2}$ equally distributed around the origin.

\begin{figure}[h]
\leavevmode \SetLabels
\L(.23*.4) $\;$\\
\endSetLabels
\begin{center}
\AffixLabels{\centerline{\epsfig{file =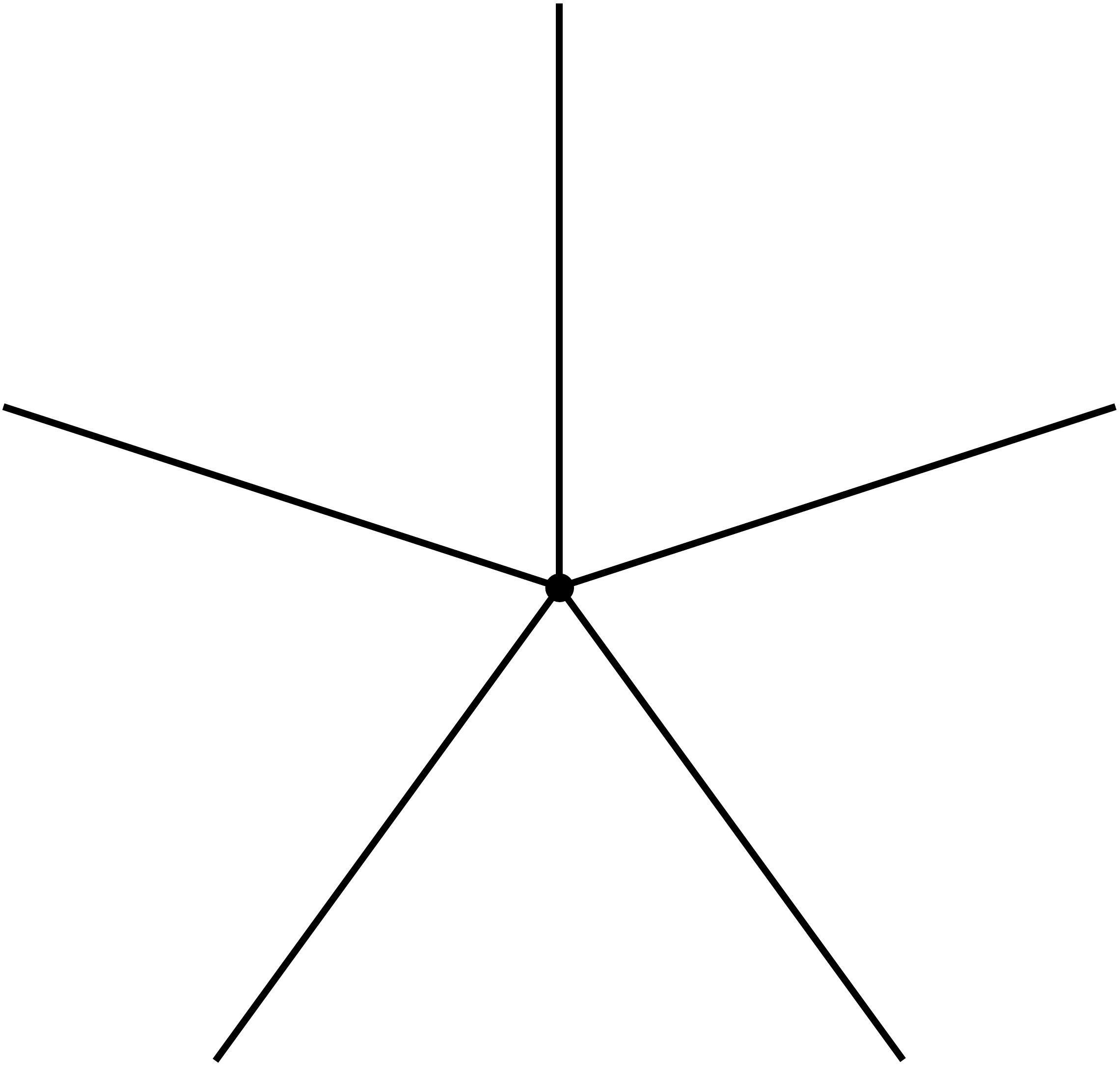,width=4.0cm,angle=0}\hspace{1.5cm} \epsfig{file =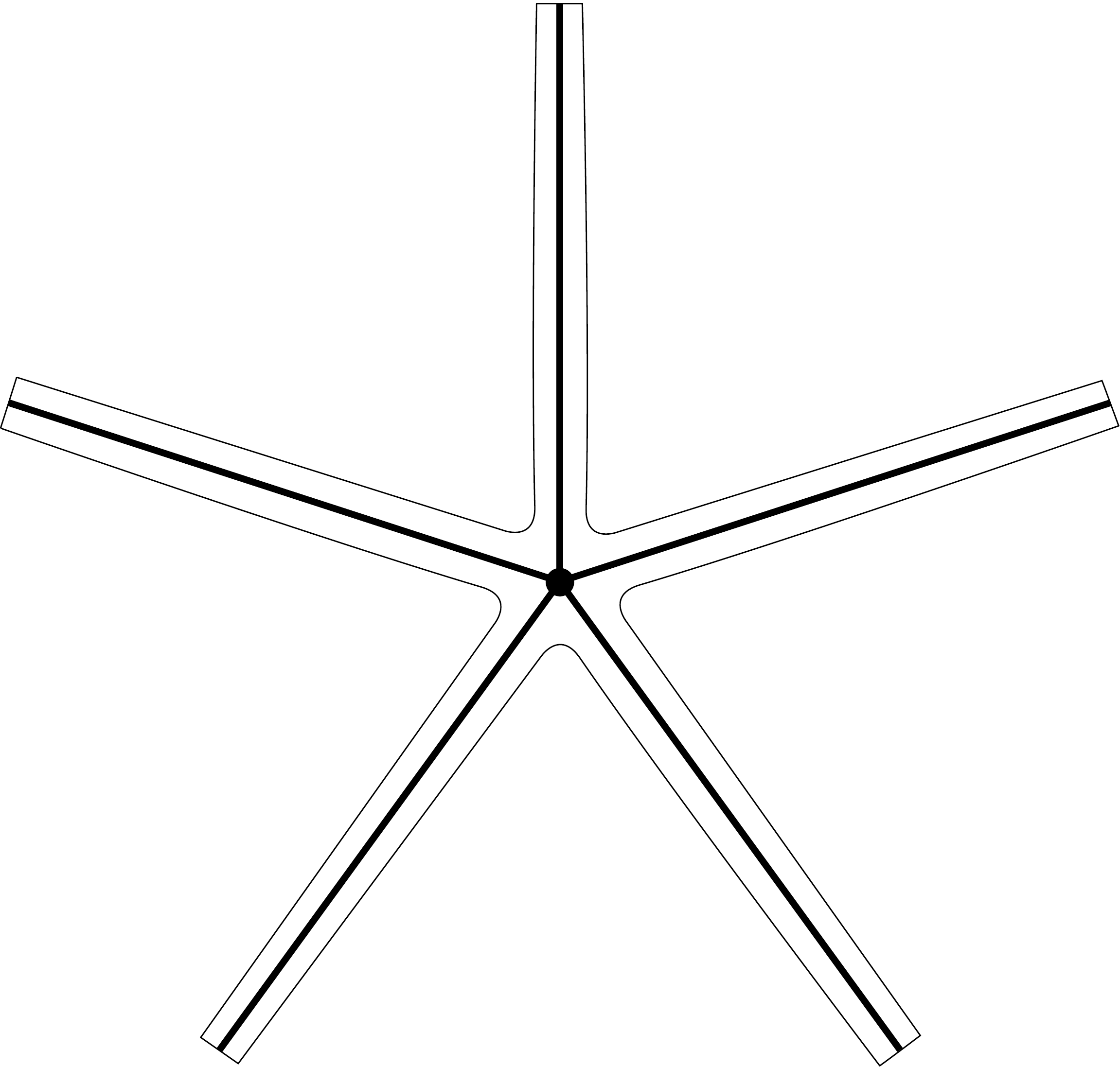,width=4.0cm,angle=0}}}
\vspace{-18pt}
\end{center}
\caption{Constructing a ribbon $n$-pod} \label{fig:eucnpods}
\end{figure}

We now fix a small $\varepsilon>0$ and consider a {\it ribbon type $n$-pod} of width $\varepsilon>0$ around our embedded $n$-pod. The ends of the ribbons are flat as is illustrated in figure \ref{fig:eucnpods}. Formally: the ribbon type $n$-pod is the intersection of the closed $\varepsilon>0$ neighborhood of the ribbon and a collection of closed half-planes defined as follows. For each euclidean ray supporting a half-edge, we consider the lower half space delimited by the unique line perpendicular to the ray and distance $\frac{1}{2}$ from the origin. The flat ends of each individual ribbon are euclidean segments of length $2 \varepsilon$. Two half ribbons can be glued geometrically in the natural way to (locally) obtain a picture of a surface with boundary (see figure \ref{fig:ribbons}).

\begin{figure}[h]
\leavevmode \SetLabels
\L(.23*.4) $\;$\\
\endSetLabels
\begin{center}
\AffixLabels{\centerline{\epsfig{file =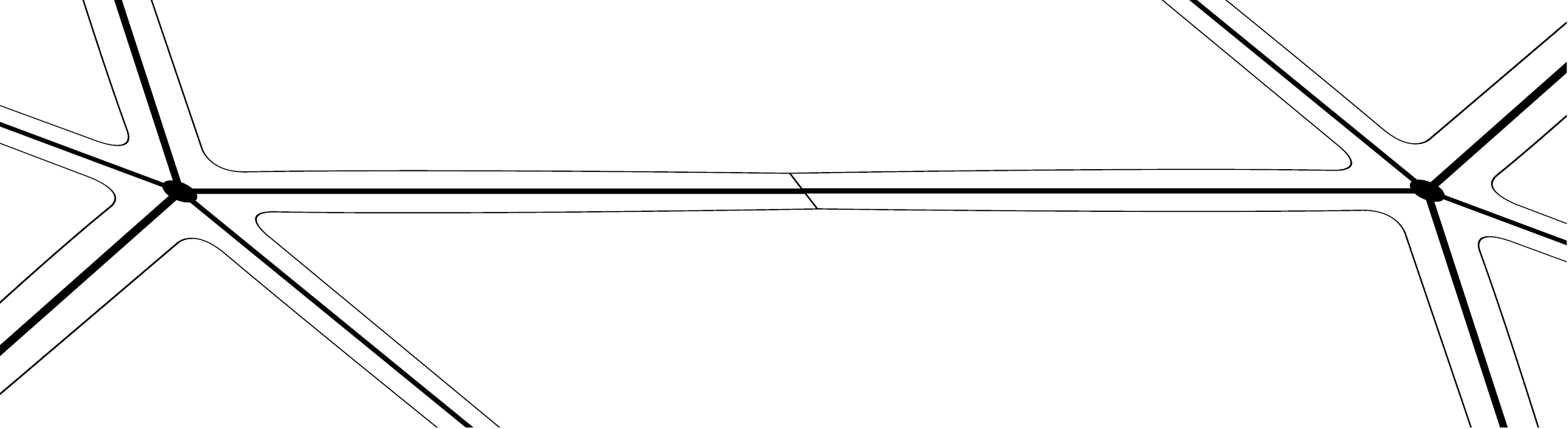,height=3.0cm,angle=0}\hspace{.5cm} \epsfig{file =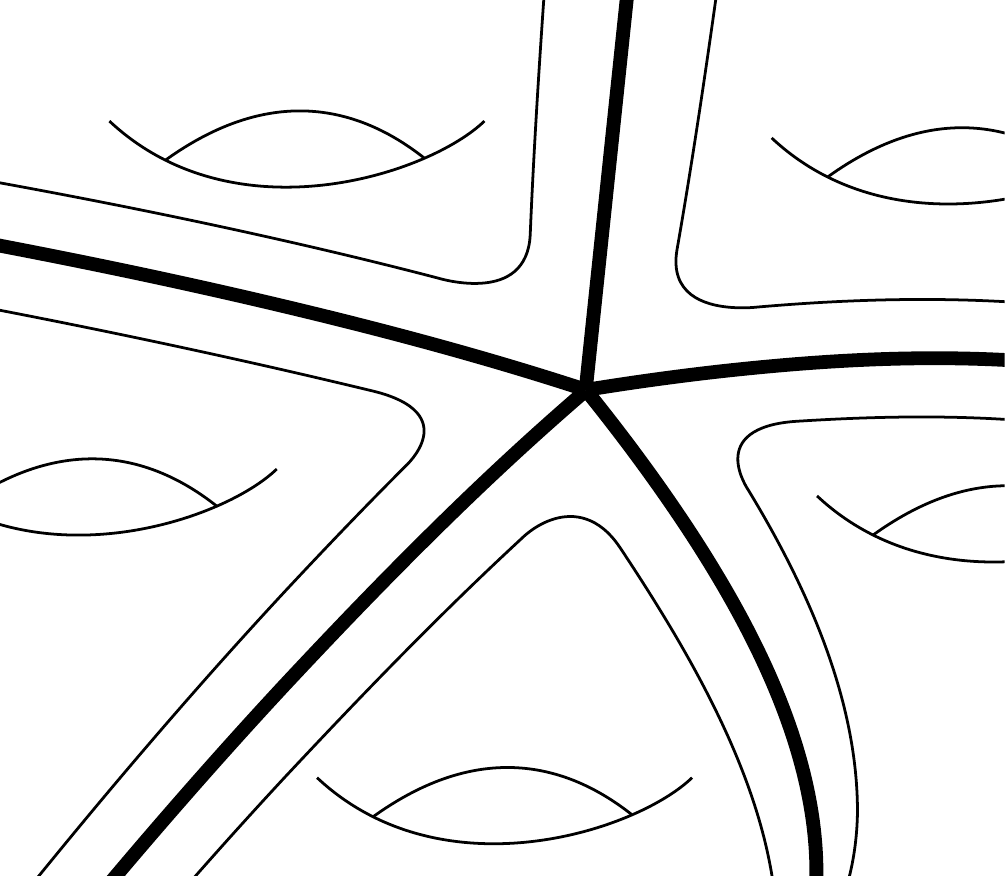,height=3.0cm,angle=0}}}
\vspace{-18pt}
\end{center}
\caption{The construction of $S_{\varepsilon.g_n}$ and its embedding in $S_{g_n}$} \label{fig:ribbons}
\end{figure}

Going back to the pairings induced by the original embedding, we perform this individual gluing on each of the pairings, and the result is a surface with boundary. Observe that the surface is homeomorphic to a closed neighborhood of the graph embedded in the surface. We denote this surface $S^\varepsilon_{g_n}$. We shall observe a few things about this surface.

First of all, thinking of $S^\varepsilon_{g_n}$ as a topological subsurface of $S_{g_n}$, observe that all the boundary curves of the embedded subsurface $S^\varepsilon_{g_n}$ are trivial (otherwise the embedding would not have minimal genus).

Via this construction, there is a natural projection (at the level of homotopy classes) from cycles in $\Gamma_n$ (and curves in $S^\varepsilon_{g_n}$) to curves in $S_{g_n}$. We are interested in minimal length cycles of $\Gamma_n$ (those of length $3$) that project to nontrivial and homotopically distinct curves on $S_{g_n}$. We now need to count them. Consider a cycle on $\Gamma_n$ and consider it on $S^\varepsilon_{g_n}$ via our original embedding. As it passes through a vertex, it enters through one half edge, exits through another and separates the ribbon type $n$-pod into $2$ pieces separating the remaining half-edges into two sets (one of which can be empty). This is illustrated in figure \ref{fig:localcurve}.
\begin{figure}[h]
\leavevmode \SetLabels
\L(.23*.4) $\;$\\
\endSetLabels
\begin{center}
\AffixLabels{\centerline{\epsfig{file =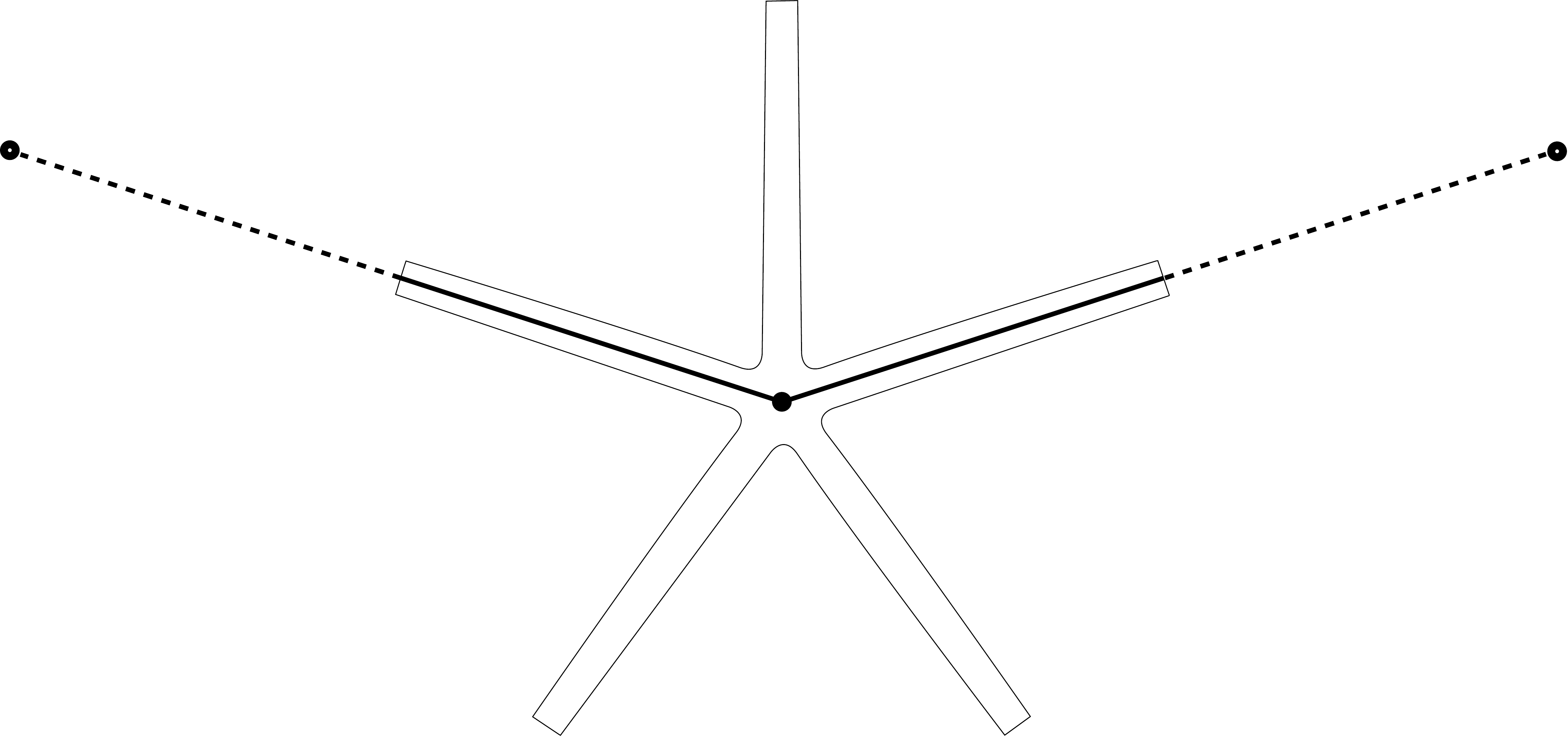,width=8.5cm,angle=0}}}
\vspace{-18pt}
\end{center}
\caption{A local picture of a short cycle of $\Gamma_n$ in $S_{g_n}$} \label{fig:localcurve}
\end{figure}

Consider the subset of all length $3$ cycles which contain a vertex on which the cycle separates the other half-edges of the associated $n$-pod into two non-empty sets. Now we claim that any such cycle projects to a {\it nontrivial} curve on $S_{g_n}$ and that any two such cycles are non-isotopic. To see this observe that any such cycle intersects another such cycle in exactly one point. Via the bigon criterion, both of these cycles are essential and in fact homologically nontrivial. Furthermore, observe that any two distinct such cycles must intersect at least one other such cycle differently, and as such are homologically, and thus non-isotopic. We can also remark that any two such cycles intersect at most once.

{\it Remark.} Although this doesn't play a role in what follows (because we are concerned with lower bounds), the cycles described above constitute the full set of short cycles of $\Gamma_n$ that project to pairwise distinct nontrivial curves on $S_{g_n}$. To see this, observe that if a short cycle fails to be in this set, then via our embedding into $S_{g_n}$, it is freely homotopic to a curve that is disjoint from the embedding of $\Gamma_n$. If it were to be nontrivial on $S_{g_n}$, then by cutting along this curve, one obtains an embedding of $\Gamma_n$ into a surface a surface of smaller genus than $S_{g_n}$, a contradiction.

\begin{figure}[h]
\leavevmode \SetLabels
\L(.23*.4) $\;$\\
\endSetLabels
\begin{center}
\AffixLabels{\centerline{\epsfig{file =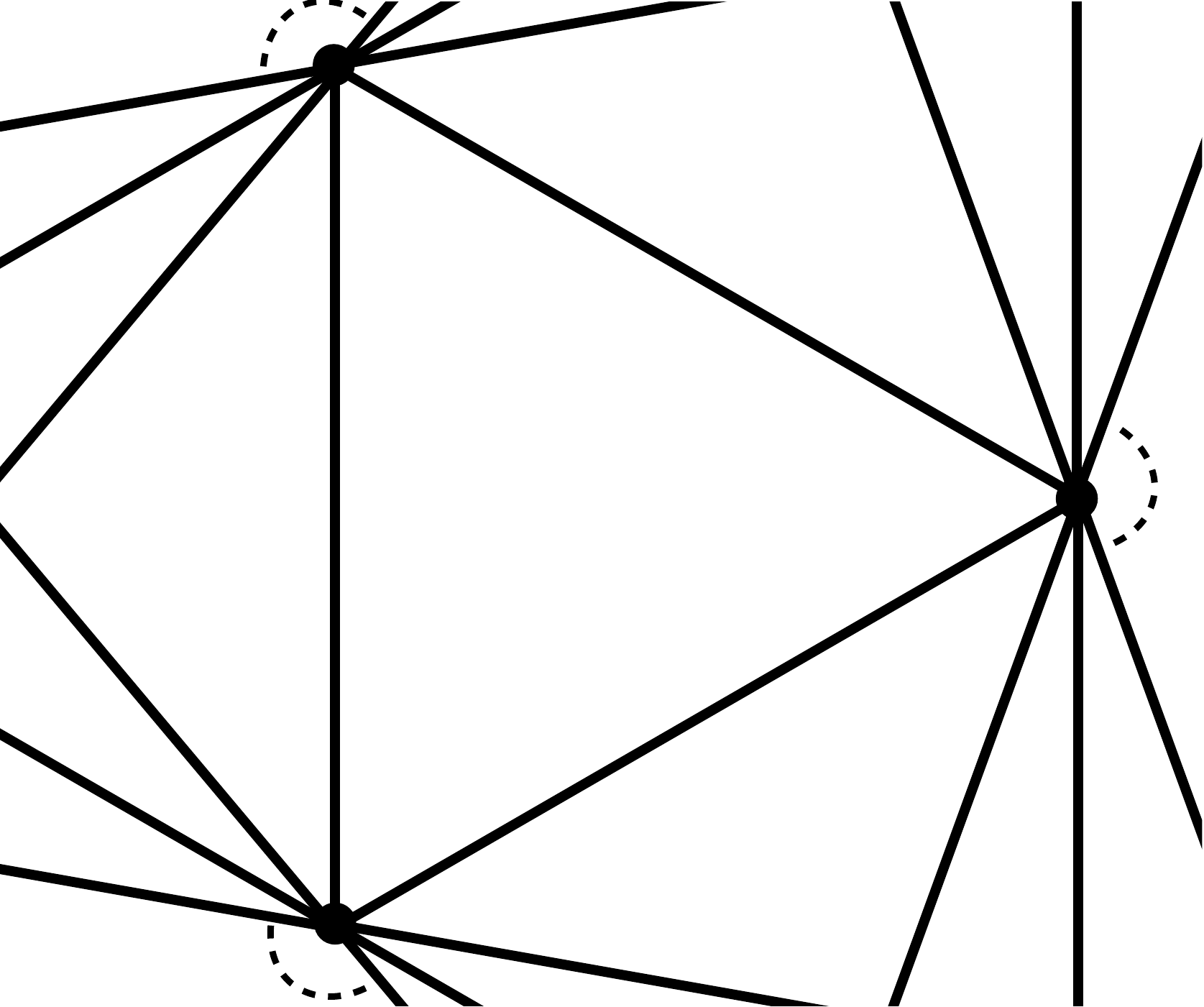,width=5.5cm,angle=0}}}
\vspace{-18pt}
\end{center}
\caption{A nontrivial cycle in $\Gamma_n$ which projects to a trivial curve in $S_{g_n}$} \label{fig:orientedpants}
\end{figure}
We can now count these cycles. Note that any two edges of a short cycle determine the cycle completely as this determines the $3$ vertices of the cycle. Around each vertex there are $\frac{n(n-3)}{2}$ choices of pairs of distinct edges joined in the vertex that separate the remaining edges of the vertex into two non-empty sets. We can make this number of choices for each vertex, and each cycle will have been counted (at most) $3$ times. The total number of cycles $N$ is thus at least
$$
N\geq  \frac{n^2(n-3)}{6}.
$$
These cycles determine a fixed set of (free) homotopy classes of curves on $S_{g_n}$ and it is the set we are to going to find a geometric structure for which they are all represented by systoles. We denote this set of free homotopy classes by $\curves$.

We now return to our surface with boundary $S^\varepsilon_{g_n}$. Each of the cycles on this surface admits a minimal length geodesic in its free homotopy/isotopy class. Often one refers to the minimal length of an isotopy class as the length of an isotopic geodesic of minimal length.

{\it Claim.} Each of the minimal length geodesics has length $\in [3-6\varepsilon, 3]$. 

{\it Proof of claim.}
The upper bound is obvious as each free homotopy class is represented by a cycle of length $3$, passing through $3$ vertices $v_1,v_2,v_3$. For the lower bound, observe that the distance between two of the vertices on $S^\varepsilon_{g_n}$ is exactly $1$. Now consider points $p_1,p_2,p_3$ along a minimal geodesic $\gamma$ that are of minimal distance to $v_1,v_2,v_3$. By construction, these distances are all less than $\varepsilon$ (any curve freely homotopic to the cycle must traverse each of the three ``corridors"). Consider the three subarcs of $\gamma$ between the points $p_1,p_2,p_3$, and denote their lengths $\ell_1, \ell_2, \ell_3$. Now by concatenating the paths, and using the fact that the vertices are at distance $1$ from each other, we see that 
$$
1 \leq \ell_k + 2 \varepsilon
$$
for $k=1,2,3$. It follows that 
$$
\ell(\gamma)= \ell_1+\ell_2+\ell_3 \geq 3 - 6\varepsilon.
$$
This proves the claim.\\

Our next objective is to modify the metric so that the curves in $\curves$ all have length exactly $3$. We begin by choosing, one for each of our isotopy classes, a minimal length geodesic. If necessary, we can paste arbitrarily thin euclidean cylinders on the boundary of the surface, to ensure that all minimal geodesics we have chosen do not touch the boundary. Observe that any {\it other} nontrivial isotopy class on the surface has minimal length at least $4$ (minus something close to $0$ when $\varepsilon$ is small). Via \cite{frhasc82}, any two minimal geodesics intersect minimally among all representatives in their respective free homotopy classes. In particular, any two distinct and freely homotopic minimal length geodesics are disjoint. It follows that there exists a $\varepsilon'>0$ (with $\varepsilon'<\varepsilon$) such that the closed $\varepsilon'>0$ neighborhood of one our chosen minimal length geodesics only (completely) contains other minimal length geodesics that are in the same free homotopy class. As such, we can consider an $\varepsilon'>0$ neighborhood of the full set of chosen minimal length geodesics for which we are sure that any two minimal length geodesics have the same ``combinatorics" if and only if they are isotopic. By {\it combinatorics} we mean that they pass through the same $\varepsilon'$ corridors (strips of total length $\sim 1$ and of width $\varepsilon'$). We consider this new surface with boundary $S^{\varepsilon'}_{g_n}$. Observe that the number of boundary components may have increased, but the genus remains the same (topologically we've just cut disjoint disks out of $S^\varepsilon_{g_n}$ to obtain $S^{\varepsilon'}_{g_n}$).

We can now proceed to the modification of the metric on $S^{\varepsilon'}_{g_n}$. Consider the set of {\it maximal length} geodesics among our set of minimal length geodesics representing $\curves$.

To illustrate the construction, let us suppose that there is only one curve of maximal length, say $\gamma$, of length $L \leq 3$. We can insert a euclidean cylinder of width $\omega$ and with both boundary lengths $\ell(\gamma)$ along $\gamma$ while respecting how the two copies of $\gamma$ were pasted together. (We will call this {\it grafting} along $\gamma$, although the term is generally used for this type of construction along geodesics of hyperbolic surfaces.) Observe that this process {\it increases} the minimal length of any isotopy class of a curve that crosses $\gamma$. To see this, suppose this was not the case, i.e., there is a curve $\delta$ on the modified surface of shorter length than a given minimal length geodesic on the original surface. Consider the subarc $c$ of $\delta$ that passes through the cylinder. Consider the projection $c'$ of $c$ to the basis of the cylinder as in figure \ref{fig:graft}. By construction $\ell(c)>\ell(c')$. Now we ``project" $\delta$ to the original surface, by concatenating the image of $c'$ and the arc of $\delta$ that was disjoint from the cylinder to obtain a curve $\tilde{\delta}$ on the original surface of strictly shorter length than $\ell(\delta)$, a contradiction.

\begin{figure}[h]
\leavevmode \SetLabels
\L(.8*.38) $\gamma$\\
\L(.683*.321) $c'$\\
\L(.281*.235) $c'$\\
\L(.34*.72) $\delta$\\
\L(.279*.421) $c$\\
\L(.743*.63) $\tilde{\delta}$\\
\endSetLabels
\begin{center}
\AffixLabels{\centerline{\epsfig{file =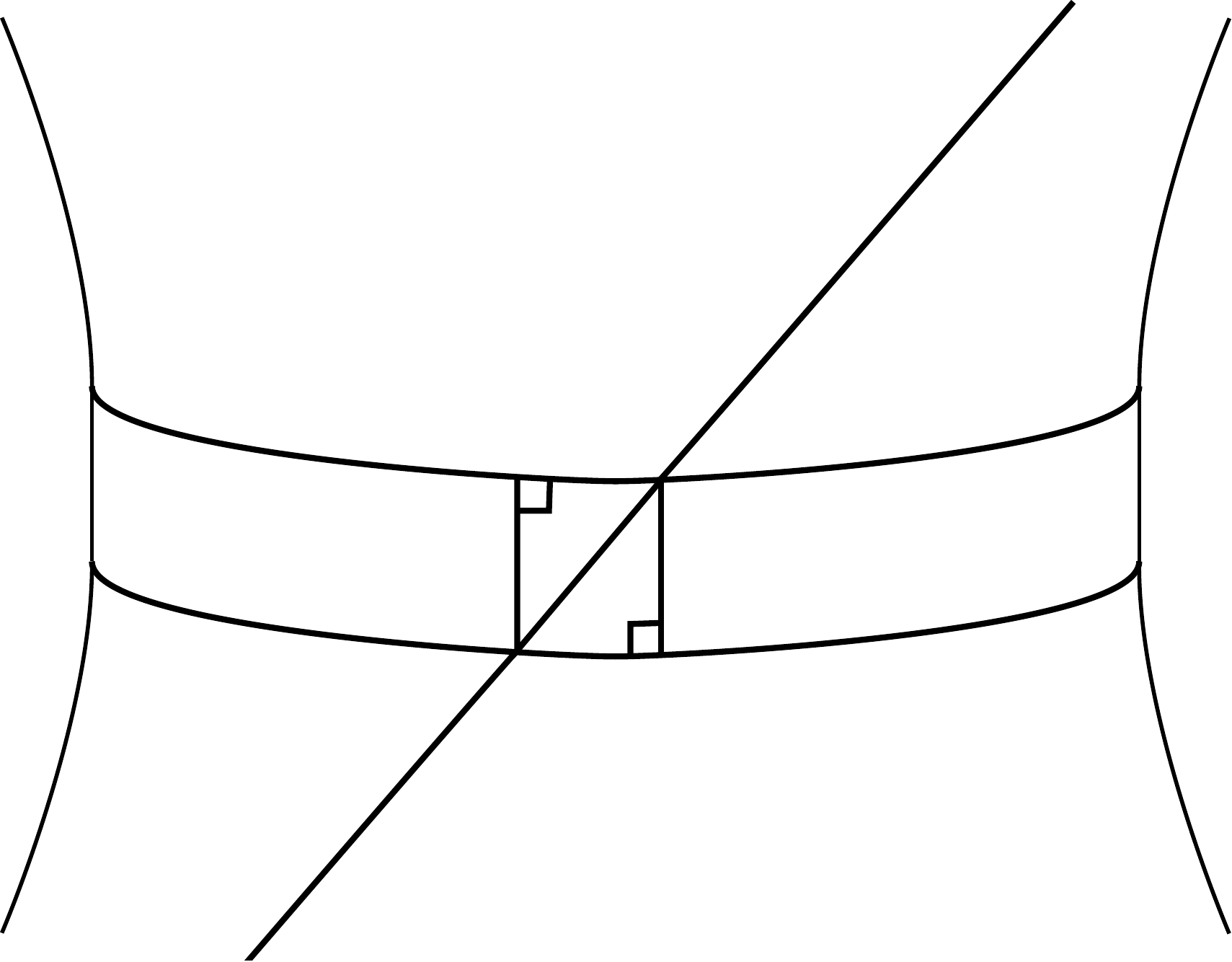,width=5.0cm,angle=0}\hspace{1.0cm} \epsfig{file =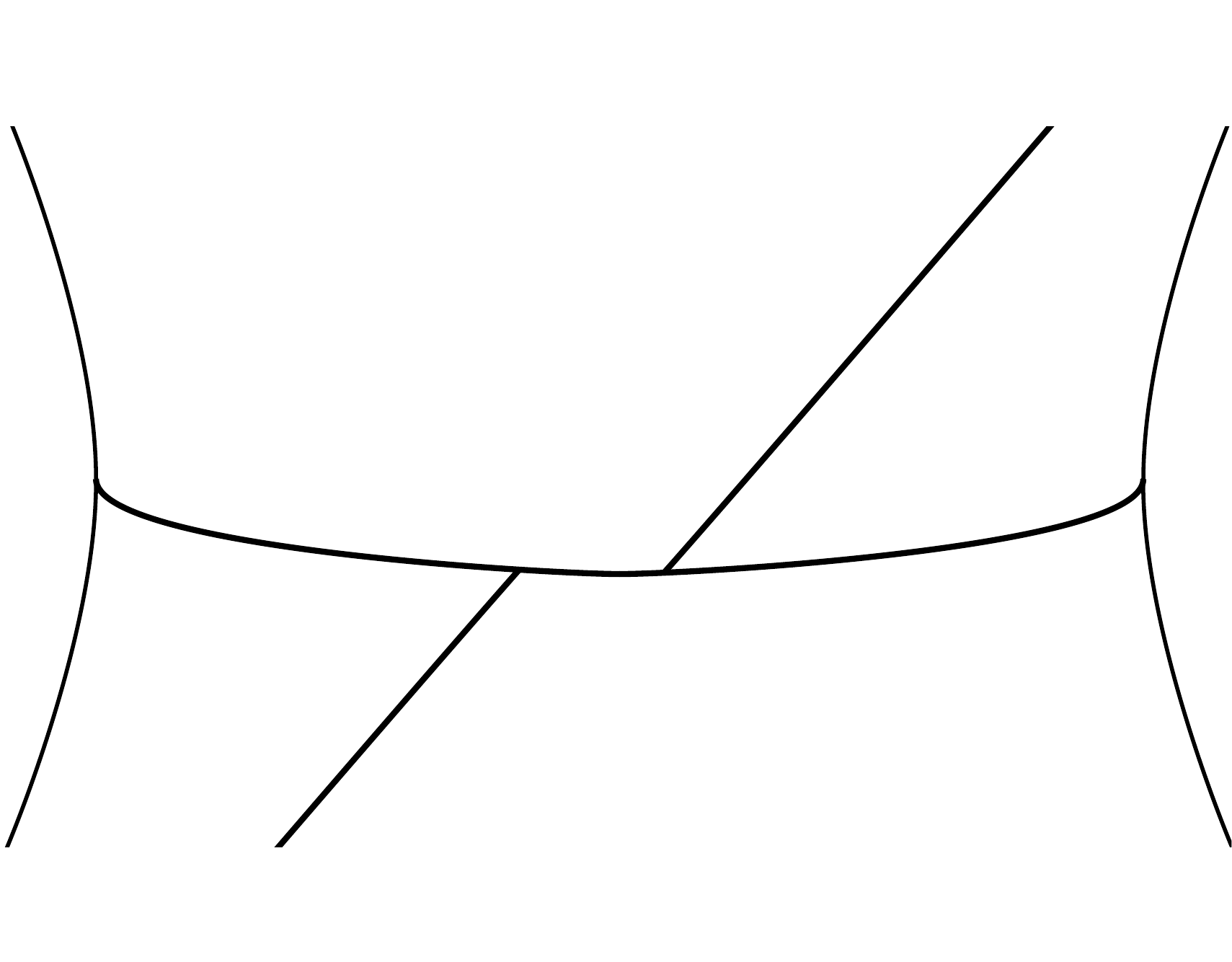,width=5.0cm,angle=0}}}
\vspace{-18pt}
\end{center}
\caption{Grafting along $\gamma$} \label{fig:graft}
\end{figure}
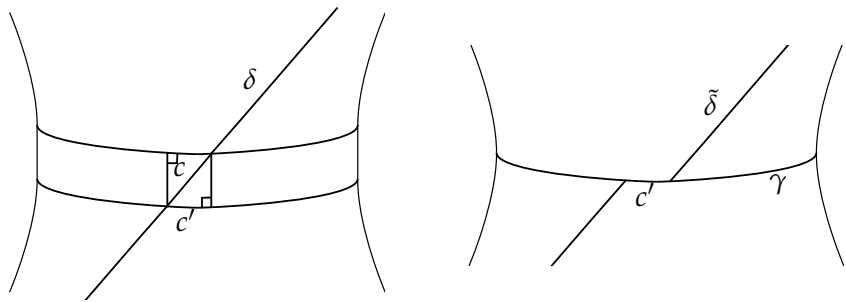

Also observe that via grafting along $\gamma$,  the minimal length of the isotopy class of $\gamma$ does not change. In fact, we obtain a cylinder of minimal length geodesics parallel to the two copies of $\gamma$. Also observe, that in light of the fact that minimal geodesics intersect minimally, the minimal length of {\it any} isotopy class that does not intersect $\gamma$ remains unchanged during this process. As $\ell(\gamma)\leq 3$, and any curve going through the inserted cylinder has length strictly less than the width of the cylinder, there exists a width $\omega_0< 3$ for which there is at least one other isotopy class among our selected classes with minimal length exactly $\ell(\gamma)$.

We now imitate this construction, and repeat it iteratively, but for more than one curve. Specifically, consider a set $\gamma_1,\hdots,\gamma_k$ of curves with maximal minimal length among our chosen classes on a surface with boundary $S^i$, obtained after $i$ metric modifications of $S^{\varepsilon'}_{g_n}$. We consider an $\tilde{\varepsilon}$ neighborhood around them, again as in the original construction of $S^{\varepsilon'}_{g_n}$. This gives us a sub-surface of $S^i$ with boundary curves $\delta_1,\hdots,\delta_b$. We insert a euclidean cylinder of width $\omega$ along each of the $\delta_k$. As above this process will make any curve that intersects at least of one of the $\gamma_k$  increase. And again, if there is at least one free homotopy class in our set of chosen classes that intersects of these curve, because the length of all the $\gamma_k$ is less or equal to $3$, there exists an $\omega_0$ for which a new homotopy class now has minimal length equal to $L$.

We now observe the following: between any two free homotopy classes $h$ and $\tilde{h}$ in our chosen set, there exist a sequence of $h^k\in \curves$ with $h^0=h$, $h^p=\tilde{h}$ and $\ii (h^i,h^{i+1})=1$. In particular, this guarantees that the above process leaves no curve in $\curves$ isolated and finishes in at most $N$ steps where $N$ is the cardinality of $\curves$.

We can now summarize the result of the construction: via a finite number of insertions of cylinders into a surface with closed boundary, we now have a surface $\tilde{S}$ of genus $g_n$, also with closed boundary which has $N$ distinct free homotopy classes of curves of minimal length $L$ and where all other (non-peripheral) curves have length at least $\sim 4$. We now complete the description of the metric by using a well-known trick: for each of the boundary curves $\delta \subset \partial \tilde{S}$, we glue a round hemisphere of equator length $\ell(\delta)$. On the resulting surface $S$, it is easy to see that {\it any} minimal geodesic will be completely disjoint from any of the hemispheres and will lie completely in the subsurface $\tilde{S}$.

In particular, the systoles we constructed in $\tilde{S}$ are systoles of $S$. As $n\sim \sqrt{g_n}$, we have $N\sim g_n^{\sfrac{3}{2}}$ and the resulting surface has the desired properties.
\end{proof}

\begin{remark}\label{rem:gint}
There is another construction of a surface, much simpler, that also underlies the difference between hyperbolic and non-hyperbolic surfaces. Consider a single ribbon $n$-pod (for some small $\varepsilon>0$ with $n=4m$, and pair opposite edges in the obvious way. The result is a surface of genus $g=m= \sfrac{n}{4}$ with one boundary component. As above, the curves from the original $n$-pod are of length exactly $1$, and because $n$ is even, they are smooth curves on the surfaces. These curves are in fact systoles: any curve with more complicated combinatorics has length at least $\sim 2$ and any curve that stays in a single ribbon has length at least $1$. (Each of these curves has a family of parallel isotopic systoles as well.) As above, we glue a euclidean hemisphere to the single boundary curve of the surface. The interest in this construction is that in the center vertex of the embedded $n$-pod, there are exactly $n= 4g$ non-isotopic systoles that intersect in a single point. This is again in striking difference to what is possible for hyperbolic surfaces in view of Lemma \ref{lem:angle}.
\end{remark}
\addcontentsline{toc}{section}{References}
\bibliographystyle{Hugo}
\def\cprime{$'$}

\end{document}